\documentclass [10pt]{amsart}
\usepackage{amsmath,amsfonts,amsthm,amssymb,pxfonts}
\usepackage{Tabbing}
\usepackage{lastpage}
\usepackage{mathtools}
\usepackage{esint}
\usepackage{color}
\usepackage{mathrsfs}
\usepackage{hyperref}
\newtheorem{theorem}{Theorem}%[section]
\newtheorem{definition}[theorem]{Definition}

\newtheorem{lemma}[theorem]{Lemma}

\newtheorem{corollary}[theorem]{Corollary}
\newtheorem{proposition}[theorem]{Proposition}

\DeclareMathOperator{\ad}{ad}

\begin{document}

\title[Sobolev Stability of plane wave solutions to the nonlinear Schr\"{o}dinger equation]{Sobolev Stability of plane wave solutions to the nonlinear Schr\"{o}dinger equation}

\author{Bobby Wilson }
\address{}
\thanks{This is in partial fulfillment of the author's requirements for the Doctor of Philosophy degree in Mathematics at the University of Chicago.  The author was partially supported by the NSF through the RTG grant \#1246999 at the University of Chicago during the completion of this work.}

\subjclass[2010]{37K45, 37K55}

\keywords{Stability of plane wave solutions to the cubic nonlinear Schr\"{o}dinger equation}

\begin{abstract}
This article explores the questions of long time orbital stability in high order Sobolev norms of plane wave solutions to the NLSE in the defocusing case.
\end{abstract}

\maketitle 

\tableofcontents

%%%%%%%%%%%%%%%%%%%%%%%%%%%%%%%%%%%%%%%%%%%%%%%%%%%%%%%%%%%%%%
%%%%%%%%%%%%%%%%%%%%%%%%%%%%%%%%%%%%%%%%%%%%%%%%%%%%%%%%%%%%%%
		\section{Introduction}\label{sec.intro}

Consider the periodic nonlinear Schr\"{o}dinger equation
\begin{align}\label{p-nls0}
i\partial_t{\psi} &= \Delta \psi + \lambda|\psi|^{2p} \psi
\end{align}
where $p \in \mathbb{N}$, $x\in\mathbb{T}^d$ and $\Delta$ is the standard Laplace-Beltrami operator.  We wish to investigate the orbital stability of plane-wave solutions to \eqref{p-nls0}. % under small, generic $H^s(\mathbb{T}^d)$ perturbations. 
For $m \in \mathbb{Z}^d$, let $w_m(x,0):= \varrho e^{im\cdot x}$ be the initial datum concentrated at the $m$th mode. We will denote by $w_m(x,t)$ the plane-wave solution to equation \eqref{p-nls0} with initial datum $w_m(x,0)$.  We will show that for any $K \in \mathbb{N}$, there exist $s_0$ and $\varepsilon_0$ so that any solution $\psi$ to \eqref{p-nls0} with initial datum that is $\varepsilon$-close to $w_m(x,0)$ in $H^s(\mathbb{T}^d)$, for $\varepsilon<\varepsilon_0$ and $s>s_0$,  will meet the condition 
\begin{align*}
\inf_{\varphi\in\mathbb{R}}\|e^{-i\varphi}e^{-im\cdot \bullet}w_m(\bullet,t)-e^{-im\cdot \bullet}\psi(\bullet,t)\|_{H^s(\mathbb{T}^d)} <
\varepsilon C(K,s_0,\varepsilon_0)
\end{align*}
 for $t< \varepsilon^{-K}$.
 Here $H^s(\mathbb{T}^d)$ is the Sobolev space.

Much has been written about this topic as outlined in the paper by Faou, Gauckler and Lubich \cite{FGL}.  For instance, instability has been demonstrated in low regularity cases by Christ, Colliander, and Tao \cite{CCT031}, \cite{CCT032} ($d=1$, $s<0$), Carles, Dumas and Sparber \cite{CDS10} ($s<0$), and Hani \cite{H11} ($0<s<1$, $p=1$).  On the other hand, for $d=1$ and $s=1$, there are stability results that can be found in Gallay and Haragus \cite{GH071}, \cite{GH072} and Zhidkov \cite{Z01}.

Results in the cubic case of our setting ($d\geq 1$, $s>1$) include results on the growth of the Sobolev norm of solutions and instability near $0$ by Bourgain \cite{B96}, Colliander, Keel, Staffilani, Takaoka and Tao \cite{CKSTT10} and Guardia and Kaloshin \cite{GK12}.

This paper, along with \cite{FGL}, uses the theory of Birkhoff normal forms in the same manner as Bambusi and Gr\'{e}bert \cite{B07, BG, G07} and Gauckler and Lubich \cite{GL10} in which the theory was applied to a modified cubic NLS, which requires high regularity, $s\gg 1$.  As opposed to proving an instability result, we show long time orbital stability of plane wave solutions to \eqref{p-nls0} in the defocusing case.

We will emulate the argument presented in \cite{FGL} using the theory of Birkhoff normal forms presented in \cite{BG}.  In \cite{FGL}, they prove
\begin{theorem}
Let $\rho_0>0$ be such that $1-2\lambda \rho_0^2>0$, and let $N>1$ be fixed arbitrarily.  There exists $s_0>0$, $C \geq 1$ and a set of full measure $\mathcal{P}$ in the interval $(0,\rho_0]$ such that for every $s\geq s_0$ and every $\rho \in \mathcal{P}$, there exists $\varepsilon_0$ such that for every $m \in \mathbb{Z}^d$ the following holds: if the initial data $u(\bullet, 0)$ are such that
\begin{align*}
\|u(\bullet,0)\|_{L^2}=\rho \hspace{.5cm} \mbox{ and } \hspace{.5cm} \| e^{-im \cdot \bullet}u(\bullet,0)-u_m(0)\|_{H^s} = \varepsilon \leq \varepsilon_0
\end{align*}
then the solution of \eqref{p-nls0} (with $p=1$) with these initial data satisfies
\begin{align*}
\|e^{-im\cdot \bullet} u(\bullet,t)-u_m(t)\|_{H^s} \leq C \varepsilon \mbox{ for } t \leq \varepsilon^{-N}
\end{align*}
\end{theorem}
We prove the same result for any $p\in\mathbb{N}$, namely
\begin{theorem}\label{mainthm}
Let $L_0>0$ be such that $1-2p\lambda L_0^{p}>0$, and let $K>1$ be fixed arbitrarily.  There exists $s_0>0$, $C \geq 1$ and a set of full measure $\mathcal{P}$ in the interval $(0,L_0]$ such that for every $s\geq s_0$ and every $L \in \mathcal{P}$, there exists $\varepsilon_0$ such that for every $m \in \mathbb{Z}^d$ the following holds: if the initial data $u(\bullet, 0)$ are such that
\begin{align*}
\|u(\bullet,0)\|^2_{L^2}=L \hspace{.5cm} \mbox{ and } \hspace{.5cm} \| e^{-im \cdot \bullet}u(\bullet,0)-u_m(0)\|_{H^s} = \varepsilon \leq \varepsilon_0
\end{align*}
then the solution of \eqref{p-nls0} with these initial data satisfies
\begin{align*}
\|e^{-im\cdot \bullet} u(\bullet,t)-u_m(t)\|_{H^s} \leq C \varepsilon \mbox{ for } t \leq \varepsilon^{-K}
\end{align*}
\end{theorem}
In essence, we show that the phenomena that allows for stability in the case $p=1$ are present for every $p\in \mathbb{N}$.   We will demonstrate a more transparent and generalized argument than what has been shown before. One aspect that is made apparent is that the techniques used in \cite{FGL}, for the case $p=1$, can be applied to the case $p>1$ while examining the vector field of the normalized Hamiltonian as we do in the final two sections of this paper.  This examination reaffirms that the stability is derived entirely from the type of linear combinations of the frequencies that are degenerate and the algebraic properties of the nonlinearity.

We should not readily expect this type of extension from the $p=1$ case.  For the one-dimensional NLSE, the references mentioned above and \cite{GK14} show us that the $p=1$ case and the $p>1$ case exhibit very different phenomena.  An important example is the fact that the one-dimensional NLSE with $p=1$ is completely integrable.  Thus, it is not obvious that the same results should follow directly from the arguments made in \cite{FGL}.

In fact, after employing the normal form change of variables, the lower degree terms that remain can either be grouped into terms that preserve, as they are called in \cite{FGL},  ``super-actions":
\begin{align*}
\sum_{|n|^2=q} |z_n|^2
\end{align*}
or are small enough to be grouped with the high degree terms which determine the extent to which the solution remains close to the orbit of $w_m$.

		\section{Functional Setting}
We first establish a setting in which to prove Theorem \ref{mainthm}.  Similar definitions, as well as the proof of the lemmas in this section, appear in \cite{BG}.
\begin{definition}
For $x=\{x_n\}_{n \in \mathbb{Z}^d}$, define the standard Sobolev norm as
\begin{align*}
\|x\|_s := \sqrt{ \sum_{n\in \mathbb{Z}^d}|x_n|^2 \langle n \rangle^{2s}   } 
\end{align*}
Define $H^s$ as
\begin{align*}
H^s:= \left\{ x= \{x_n\}_{n\in \mathbb{Z}^d} \,\left| \, \|x\|_s<\infty \right. \right\}
\end{align*} 
\end{definition}

Consider a vector-valued homogeneous polynomial, $X$, of degree $\ell$ be written as 
\begin{align*}
X(z) = \sum_{\|\alpha\|_{\ell_1}=\ell} X_{\alpha} z^{\alpha}
\end{align*}
We will denote the majorant of $X$ by 
\begin{align*}
\tilde{X}(z):=\sum_{\|\alpha\|_{\ell_1}=\ell} |X_{\alpha}| z^{\alpha}
\end{align*}

\begin{definition}[Tame Modulus] \label{tamemod}
Let $X$ be a vector-valued homogeneous polynomial of degree $\ell$.  $X$ is said to have $s$-tame modulus if there exists $C>0$ such that
\begin{align*}
\left\| \tilde{X}(z^{(1)},...,z^{(\ell)}) \right\|_s \leq C \frac{1}{\ell} \sum_{k=1}^{\ell} \|z^{(1)}\|_{\frac{d+1}{2}}\cdots \|z^{(k-1)}\|_{\frac{d+1}{2}}\|z^{(k)}\|_s\|z^{(k+1)}\|_{\frac{d+1}{2}}\cdots \|z^{(\ell)}\|_{\frac{d+1}{2}}
\end{align*}
for all $z^{(1)},...,z^{(\ell)} \in H^s$.  The infimum over all $C$ for which the inequality holds is called the tame $s$-norm and is denoted $|X|_s$.
\end{definition}

Definition \ref{tamemod} is an extension of Definition 2.2 in \cite{BG} with $d=1$.  The inequality is related to the following property of Sobolev spaces:  Consider two functions $u,v \in H^s(\mathbb{T}^d)$ for $s > \frac{d}{2}$, then by Leibnitz rules and Sobolev embedding we have
\begin{align*}
\|u\cdot v\|_{H^s(\mathbb{T}^d)}  &\leq C_s\big( \|u\|_{H^s(\mathbb{T}^d)} \|v\|_{\infty}+\|u\|_{\infty} \|v\|_{H^s(\mathbb{T}^d)} \big) \\
&\leq C_{s,t}\big( \|u\|_{H^s(\mathbb{T}^d)} \|v\|_{H^t(\mathbb{T}^d)}+\|u\|_{H^t(\mathbb{T}^d)} \|v\|_{H^s(\mathbb{T}^d)} \big) 
\end{align*}
for any $t>\frac{d}{2}$.  On the Fourier side, the product $u\cdot v$ becomes a convolution of Fourier coefficients $\hat{u} *\hat{v}$.  Therefore, we note that if $X$ is the function on sequences, $X(z^{(1)},\ldots,z^{(\ell)})=\tilde{X}(z^{(1)},\ldots,z^{(\ell)})=z^{(1)}*\ldots* z^{(\ell)}$, 
by the same logic, there exists $C(s)$ such that
$$
\|X(z^{(1)},\ldots,z^{(\ell)})\|_s\leq C(s,d)
\sum_{k=1}^{\ell} \|z^{(1)}\|_{\frac{d+1}{2}}\cdots \|z^{(k-1)}\|_{\frac{d+1}{2}}\|z^{(k)}\|_s\|z^{(k+1)}\|_{\frac{d+1}{2}}\cdots \|z^{(\ell)}\|_{\frac{d+1}{2}}
$$
which is usually called ``tame property of the $H^s$ norm'' when $d=1$ (see, for instance, \cite{LM}).  We choose $\frac{d+1}{2}$ in replace of $t$ for convenience and
in order to be consistent with \cite{BG} when $d=1$.
Of course, when $X(z^{(1)},\ldots,z^{(\ell)})$ is not a vector-valued homogeneous polynomial of degree $\ell$,
 this property might not be satisfied.  We now define two more norms on vector fields 
\begin{definition}
Let $X$ be an vector-valued analytic function from $B_s(R)$ to $H^s$
where $B_s(R)=\{ x \in H^s \,|\, \|x\|_s \leq R\}$.  We denote
\begin{align*}
\|X\|_{s,R} := \sup_{\|z\|_s \leq R} \|X(z)\|_s
\end{align*} 
\end{definition}
\begin{definition}
Let $X$ be a nonhomogeneous vector-valued polynomial and consider its Taylor expansion
\begin{align*}
X= \sum X_{\ell}
\end{align*}
where $X_{\ell}$ is homogeneous of degree $\ell$ and assume $|X_{\ell}|_s<\infty$ for all $\ell$.  For $R>0$, we define
\begin{align*}
\langle X\rangle_{s,R} := \sum_{\ell\geq 1} |X_{\ell}|_s R^{\ell} 
\end{align*}
\end{definition}

The following lemmas provide context and a theoretical foundation for which to articulate and understand the reasoning behind  the boundedness of the change of variables in Theorem \ref{normalform} and Lemma \ref{iterlemma} and the bounds on the norms of the resulting vector fields.  Let
\begin{align*}
B_s(R):= \{ x\in H^s\,|\, \|x\|_s \leq R\}
\end{align*}
\begin{lemma} \label{comp.est}
Let $H$ be a Hamiltonian and $X_H:B_s(R) \rightarrow H^s$ the associated Hamiltonian vector field.  Fix $0< r<R$, and assume that $\|X_H\|_{s,R} <\frac{r}{3}$ , and consider the time $t$ flow $\phi_t$ of $X_H$.  Then, for $|t|\leq 1$,
\begin{align*}
\sup_{\|z\|_s\leq R-\frac{r}{3}} \|\phi_t(z) -z\|_s \leq \|X_H\|_{s,R}
\end{align*}
and for any analytic vector field $Y$ one has
\begin{align*}
\|Y \circ \phi_t\|_{s,R-r} \leq 2 \|Y\|_{s,R}
\end{align*}
\end{lemma}

The next lemma is especially important for establishing the negligibility of the terms in our transformed vector that are not normalized.  We will not eliminate all nonresonant terms of small degree. Rather, we will eliminate terms so that the remaining low degree terms can be made to be as small as we want with an application of the following lemma:
\begin{lemma} 
Fix $N$, and consider the decomposition $z=\bar{z}+\tilde{z}$. Where $\bar{z}:= \{ z_n\}_{|n|\leq N}$ and $\tilde{z}:=\{ z_n\}_{|n|\geq N}$.  Let $X$ be a vector-valued polynomial with finite tame $s$-norm and assume that $X$ has a zero of order two in the variables $\tilde{z}$.  then one has
\begin{align*}
\|X\|_{s,R} \lesssim \frac{\langle X \rangle_{s,R}}{N^{s-\frac{d+1}{2}}}
\end{align*}
\end{lemma}

For any two vector fields $X$ and $Y$, let $[X,Y]$ be the standard Lie bracket and define that adjoint function
\begin{align*}
\ad_Y(X):=[Y,X].
\end{align*}
Then we have the following two lemmas necessary for managing the effect of applying the $\ad$ function on the vector field infinitely many times in the definition of the normal form change of variables. 
\begin{lemma}
For any $r<R$, one has
\begin{align*}
\langle [X,Y] \rangle_{s,R-r} \leq \frac{1}{r} \langle X\rangle_{s,R} \langle Y \rangle_{s,R}
\end{align*}
\end{lemma}

\begin{lemma}
For any $r<R$, one has
\begin{align*}
\langle \ad_Y^n(X) \rangle_{s,R-r} \leq \frac{e^n}{r^n} \langle X\rangle_{s,R} \left( \langle Y\rangle_{s,R} \right)^n
\end{align*}
\end{lemma}

		\section{Symmetry Reduction and Diagonalization}
Let us consider equation \eqref{p-nls0} with the assumption $\lambda=-1$
\begin{align}\label{p-nls}
i\partial_t{\psi} &= \Delta \psi - |\psi|^{2p} \psi
\end{align}
where $p \in \mathbb{N}$, $x\in\mathbb{T}^d$ and $\Delta$ is the standard Laplace-Beltrami operator.

By the gauge invariance of \eqref{p-nls}, it suffices to continue assuming $m=0$. In Appendix \ref{appA}, we show it is sufficient to prove that  
\begin{align*}
\|\psi(\cdot,t)-\psi_0(t)\|^2_{H^s(\mathbb{T}^d)} <\varepsilon C(N,s_0,\varepsilon_0)
\end{align*}
 for $t< \varepsilon^{-N}$.

We denote $L:=\|\psi(0)\|^2_{L^2}$ and we assume that the $H^s$ norm of the initial datum is concentrated
at the zero mode for some $s>0$, i.e. $\|\psi(0)-{\psi}_0(0)\|_s=\varepsilon$.  In order to eliminate the zero mode, we will construct a symplectic map on the Hamiltonian.  Recall that the Hamiltonian corresponding to \eqref{p-nls} is
\begin{align} \label{ham}
H:=  \sum_{k \in \mathbb{Z}^d} |k|^2 |u_k|^2+ \frac{1}{p+1}\sum_{\sum_{i=1}^{p+1} k_i= \sum_{i=1}^{p+1} h_i} u_{k_1} \dots u_{k_{p+1}}\bar{u}_{h_1}\dots \bar{u}_{h_{p+1}}.
\end{align}
 Define the symplectic reduction of $u_0$:
\begin{align*}
&\{u_k, \bar{u}_k\}_{k \in \mathbb{Z}^d} \rightarrow (L,\nu_0, \{v_k,\bar{v}_k\}_{k \in \mathbb{Z}^d \setminus \{0\}}), \\ 
&u_0 = e^{i\nu_0}\sqrt{L-\sum_{k \in \mathbb{Z}^d} |v_k|^2}, \hspace{.3cm} u_k= v_ke^{i\nu_0}, \hspace{.3cm} \forall k \in \mathbb{Z}^d\setminus \{0\}.
\end{align*}
Inserting this change of variables into \eqref{ham} we obtain
\begin{align} \label{mainham}
&\sum_{k \in \mathbb{Z}^d \setminus\{0\}} |k|^2|v_k|^2+ \frac{1}{p+1}\big(L - \sum_{k \in \mathbb{Z}^d \setminus \{0\}} |v_k|^2 \big)^{p+1}+ \big(L - \sum_{k \in \mathbb{Z}^d \setminus \{0\}} |v_k|^2 \big)^{p}\Big( \sum_{k \in \mathbb{Z}^d\setminus \{0\}} (p+1)|v_k|^2 +\frac{p}{2}(v_kv_{-k}+\bar{v}_k\bar{v}_{-k})\Big)\\
&+ \big(L - \sum_{k \in \mathbb{Z}^d \setminus \{0\}} |v_k|^2 \big)^{p-\frac{1}{2}}\sum_{k_1,k_2 \in \mathbb{Z}^d\setminus \{0\} \atop k_1+k_2\neq 0} \Big(\frac{p(p-1)}{6} (v_{k_1}v_{k_2}v_{-k_1-k_2}+c.c) +\frac{(p+1)p}{2}(v_{k_1}v_{k_2}\bar{v}_{k_1+k_2}+c.c.)\Big) \nonumber\\
&+\big(L - \sum_{k \in \mathbb{Z}^d \setminus \{0\}} |v_k|^2 \big)^{p-1}\sum_{k_i \in \mathbb{Z}^d\setminus \{0\} \atop k_1+k_2\neq k_3+k_4} \Big(\frac{p^2(p+1)}{4} (v_{k_1}v_{k_2}\bar{v}_{k_3}\bar{v}_{k_4}+c.c) +\frac{(p+1)p(p-1)}{6}(v_{k_1}v_{k_2}v_{k_3}\bar{v}_{k_4}+c.c.)\Big) \nonumber\\
&+\big(L - \sum_{k \in \mathbb{Z}^d \setminus \{0\}} |v_k|^2 \big)^{p-1}\Big(\frac{p(p-1)(p-2)}{12}\sum_{k_i \in \mathbb{Z}^d\setminus \{0\} \atop k_1+k_2\neq k_3+k_4}  (v_{k_1}v_{k_2}v_{k_3}v_{k_4}+c.c)\Big)+h.o.t. \nonumber
\end{align}
Expanding we have
\begin{align*}
&\frac{1}{p+1}L^{p+1}+ \sum_{k \in \mathbb{Z}^d \setminus\{0\}} (|k|^2+pL^p)|v_k|^2+ L^p\Big(\frac{p}{2}(v_kv_{-k}+\bar{v}_k\bar{v}_{-k})\Big) \\
&+ L^{p-\frac{1}{2}}\sum_{k_1,k_2 \in \mathbb{Z}^d\setminus \{0\} \atop k_1+k_2\neq 0} \Big(\frac{p(p-1)}{6} (v_{k_1}v_{k_2}v_{-k_1-k_2}+c.c) +\frac{(p+1)p}{2}(v_{k_1}v_{k_2}\bar{v}_{k_1+k_2}+c.c.)\Big)\\
&+ \big(- pL^{p-1}\sum_{k \in \mathbb{Z}^d \setminus \{0\}} |v_k|^2 \big)\Big( \sum_{k \in \mathbb{Z}^d\setminus \{0\}} (p+1)|v_k|^2 +\frac{p}{2}(v_kv_{-k}+\bar{v}_k\bar{v}_{-k})\Big)+\big(\frac{p}{2}L^{p-1}\big(\sum_{k \in \mathbb{Z}^d \setminus \{0\}} |v_k|^2\big)^2\big)\\
&+L^{p-1}\sum_{k_i \in \mathbb{Z}^d\setminus \{0\} \atop k_1+k_2\neq k_3+k_4} \Big(\frac{p^2(p+1)}{4} (v_{k_1}v_{k_2}\bar{v}_{k_3}\bar{v}_{k_4}+c.c) +\frac{(p+1)p(p-1)}{6}(v_{k_1}v_{k_2}v_{k_3}\bar{v}_{k_4}+c.c.)\Big)\\
&+L^{p-1}\Big(\frac{p(p-1)(p-2)}{12}\sum_{k_i \in \mathbb{Z}^d\setminus \{0\} \atop k_1+k_2\neq k_3+k_4}  (v_{k_1}v_{k_2}v_{k_3}v_{k_4}+c.c)\Big)+h.o.t.
\end{align*}

We now diagonalize the quadratic part of the Hamiltonian:
\begin{align} \label{quada}
H_0= \sum_{k \in \mathbb{Z}^d\setminus \{0\}} (k^2+L^{p}p)|v_k|^2+L^p\frac{p}{2}(v_kv_{-k}+\bar{v}_k\bar{v}_{-k})
\end{align}
which amounts to diagonalizing the matrices:
\begin{align*}
J_k = k^2 \left(\begin{array}{cc} 1 &0 \\ 0 & -1\end{array} \right)+ L^pp \left(\begin{array}{cc} 1 &1 \\ -1&-1\end{array} \right)
\end{align*}
which group together $v_k, \bar{v}_{-k}$.  We set
\begin{align*}
x_k = a_k v_k+ b_k \bar{v}_{-k}, ~ k\neq 0
\end{align*}
and in these variables
\begin{align} \label{quadb}
H_0= \sum_{k \in \mathbb{Z}^d} \frac{\Omega_k}{2} (|x_k|^2+|x_{-k}|^2)
\end{align}
with $\Omega_k= \sqrt{|k|^2(|k|^2+2pL^p)}$. 

We note that $\Omega_n=\Omega_{m}$ whenever $|n|=|m|$. Therefore it might be convenient to
group together the modes having the same frequency i.e. to denote
\begin{equation}\label{omegas}
\omega_{q}:=\sqrt{q^2(q^2+2pL^{p})},\qquad q\geq 1.
\end{equation}

Before we continue, we note a crucial feature of our Hamiltonian and the vector field in the variables $x_k, \bar{x}_k$.  Every monomial in $H(x,\bar{x})$,
\begin{align*}
x_{k_1}\cdots x_{k_p}\bar{x}_{n_1}\cdots \bar{x}_{n_q},
\end{align*}
obeys the law of {\em Conservation of Momentum}. Namely
\begin{align} \label{COM}
k_1+\cdots + k_p = n_1+\cdots + n_q
\end{align}
This property will be extremely important to the dynamics of the Hamiltonian.
		\section{Normal Form}
			
		We are now in the position to apply the theory of Birkhoff normal forms from Bambusi and Gr\'ebert \cite{BG}.  We demonstrate, for completeness, the formal argument and introduce the nonresonance condition.  After we demonstrate the normalization of our vector field, we can proceed to developing dynamical properties of the system.
	
Let us consider an auxiliary Hamiltonian $H(x)$, denote by $X_H$ the corresponding vector
field and by $\phi^t_H(x_0)$ the
time $t$ flow associated  to $H$.			
We note that 
for any vector field $Y$, its transformed vector field under the time 1 flow generated by $X_H$ is
\begin{align}\label{lie}
 e^{\ad_{X_H}}Y= \sum_{k=0}^{\infty} \frac{1}{k!} \ad_{X_H}^k Y
\end{align}
where $\ad_{X} Y:=[Y,X]$.
			\subsection{Formal Argument}
			
Consider the equation 
\begin{align} \label{form}
i\partial_t(y)_q= \omega_q (y)_q + \sum_{k\geq 2} \big(f_k(y)\big)_q
\end{align}
where for any sequence $y$ indexed by $\mathbb{Z}^d$ and $q \geq 1$
\begin{align*}
(y)_q:=
\left( \begin{array}{c} y_{n_1} \\ \cdots \\ y_{n_{k_q}} \end{array} \right)
\end{align*}
 with $k_q :=\# \{n \in \mathbb{Z}^d \,|\, |n|=q\}$. Suppose that each $f_k$ is a vector-valued homogeneous polynomial of degree $k$. 
We note that if we group together the components $x_n$ with $|n|=q$ and use the change of variables that takes \eqref{quada} to \eqref{quadb}, then we can  rewrite the vector field for \eqref{mainham} in the form of \eqref{form}.

Our aim is to use an iterative argument which puts \eqref{form} into ``normal form'' up to some predetermined degree.
As usual, at each step we use a change of variables given by a time-1 flow map
associated with a suitable Hamiltonian vector field. We proceed by demonstrating this process of normalizing the vector field in \eqref{form} at degree $K_0\geq 2$.

Let $H$ be a Hamiltonian of degree $K_0$ and consider the change of variables
\begin{align*}
y=\Phi_H(x)
\end{align*}
where $\Phi_H(x)$ is the time-1 flow map associated with the Hamiltonian vector field $X_H$.
Using the identity \eqref{lie}, one obtains
\begin{align*}
i \partial_t (y)_q= \omega_q (y)_q+  \sum^{K_0-1}_{k= 2} \big(f_k(y)\big)_q
+([X_H,\omega y ](y))_q +
  (f_{K_0}(y))_q
+ h.o.t.
\end{align*}
where $\omega y$ is the vector field with components
$$
\begin{pmatrix}
(\omega y)_n \cr \ \cr \overline{(\omega y)}_{-n}
\end{pmatrix}=
\left(  \begin{array}{cc} \omega_q & 0\\ 0 & -\omega_q \end{array}  \right) \left( \begin{array}{c} y_n\\ \bar{y}_{-n} \end{array} \right).
$$

The idea is to choose $H$ and another vector-valued homogeneous polynomial of degree $K_0$, $R_{K_0}$, in such a way that we can decompose $f_{K_0}$ as follows
\begin{align} \label{homo}
 f_{K_0}(y)=R_{K_0}(y)-[X_H,\omega y  ](y).
\end{align}
%
%\begin{align*}
%i\partial_t \left( \begin{array}{c} y_n\\ \bar{y}_{-n} \end{array} \right) = \left(  \begin{array}{cc} \Omega_n & 0\\ 0 & -\Omega_n \end{array}  \right) \left( \begin{array}{c} y_n\\ \bar{y}_{-n} \end{array} \right)+R_2(y)+ \sum_{k\geq 3} \tilde{F}_k(y)
%\end{align*}
We can find $H$ so that $R_{K_0}$ is in the kernel of the following function from the space of polynomial vector fields into itself
\begin{align*}
\ad_{\omega} (X):=[X,\omega y].
\end{align*}
Any $Y \in \ker \ad_{\omega}$ is referred to as "normal" or "resonant".
In order to correctly choose $H$ and $R_{K_0}$, we will use the theory developed in \cite{BG}.  First, let us characterize the normal terms with respect to the nonresonace condition of the frequencies $\{\Omega_n\}$.  The monomials in $f_{K_0}$ that are normal are those terms $y^{\alpha}\bar{y}^{\beta}\partial_{y_m}$, where $\alpha,\beta\in\mathbb{N}^{\infty}$ with
$\|\alpha\|_1+\|\beta\|_1=K_0$  with components $\delta_{j,m}$ ($\delta_{j,m}$ being
the Kronecker symbol), such that 
\begin{align*}
y^{\alpha}\bar{y}^{\beta}\partial_{y_m} \in \ker \ad_{\omega }.
\end{align*}
We note that
\begin{align*}
\ad_{\omega }(y^{\alpha}\bar{y}^{\beta}\partial_{y_m})= [(\alpha-\beta)\cdot \Omega-\Omega_m]y^{\alpha}\bar{y}^{\beta}\partial_{y_m}
\end{align*}
so that we can characterize $\ker \ad_{\omega }$ as 
\begin{align*}
\ker \ad_{\omega }= \mbox{span} \left\{ y^{\alpha}\bar{y}^{\beta}\partial_{y_m} \,| \,(\alpha-\beta)\cdot \Omega-\Omega_m=0  \right\}
\end{align*}
where $(\alpha-\beta)\cdot \Omega=\sum_{i \in \mathbb{Z}^d} \alpha_i \Omega_i-\sum_{i \in \mathbb{Z}^d} \beta_i \Omega_i$.
Let $X_H$ be a homogeneous vector-valued polynomial of degree $K_0$. We Taylor expand $Y$, $X_H$, and $R_{K_0}$
\begin{align*}
Y(y,\bar{y})= \sum_{\alpha,\beta,m} Y_{\alpha,\beta,m}y^{\alpha}\bar{y}^{\beta} e_m\\
X_H(y,\bar{y}) =  \sum_{\alpha,\beta,m} X_{\alpha,\beta,m}y^{\alpha}\bar{y}^{\beta} e_m\\
R_{K_0}(y,\bar{y}) =  \sum_{\alpha,\beta,m} R_{\alpha,\beta,m}y^{\alpha}\bar{y}^{\beta} e_m
\end{align*}
The homological equation \eqref{homo} becomes
\begin{align*}
 R_{\alpha,\beta,m}-(\Omega\cdot (\alpha-\beta)-\Omega_m)X_{\alpha,\beta,m}  = Y_{\alpha,\beta,m}
\end{align*}
Now we define $X_H$ and $R_{K_0}$ as follows:
\begin{align*}
&\begin{array}{l} R_{\alpha,\beta,m} := Y_{\alpha,\beta,m} \\ X_{\alpha,\beta,m}:=0 \end{array} ~\mbox{ when }~ \Omega\cdot (\alpha-\beta)-\Omega_m=0\\
&X_{\alpha,\beta,m}:=\frac{-Y_{\alpha,\beta,m}}{(\Omega\cdot (\alpha-\beta)-\Omega_m)} ~\mbox{ when }~ \Omega\cdot (\alpha-\beta)-\Omega_m\neq 0
\end{align*}

We note that through this definition $H$ will be a Hamiltonian and that this change of variables preserves conservation of momentum, \eqref{COM}.

If we define $\lambda_q := \sum_{|i| =q} \alpha_i -\beta_i$, then the expression
 $
(\alpha-\beta)\cdot \Omega-\Omega_m
$
 becomes
\begin{align*}
\sum_{q \geq 1} \lambda_q \omega_q.
\end{align*}
		\subsection{Nonresonance Condition}
Now that we have a formal characterization of resonant polynomials, we can state a normal form theorem and determine dynamical properties of our system.  Given an $M\in \mathbb{N}$  dependent on the highest degree at which we will perform a normal form reduction, we have the following condition applicable to our parameter $L$ from the definition \eqref{omegas}:

\begin{definition}[Nonresonance Condition]
There exists $\gamma=\gamma_{M}>0$ and $\tau=\tau_M >0$ such that for any $N$ large enough, one has
\begin{equation}\label{nrcond}
\left| \sum_{q\geq1} \lambda_{q}\omega_{q} \right|\geq \frac{\gamma}{N^{\tau}}\hspace{1cm}\mbox{for } \|\lambda\|_1\leq M,
\ \ \sum_{q>N}|\lambda_q|\leq 2
\end{equation}
where $\lambda\in \mathbb{Z}^{\infty}\setminus\{0\}$.
\end{definition}

The following generalization of the ``non-resonance'' result in \cite{BG} holds.
%\marginpar{citation? This is based on section 5 of Bambusi, Grebert.  I will write out the full proof though}

	\begin{theorem} \label{resthm}
For any $P>0$, there exists a set $J \subset (0,P)$ of full measure such that if $L^{p} \in J$ then for any $M>0$
% and any  $\{\ell_j\}$, $\{n_j\}_{j=1}^K$
 the Nonresonance Condition holds. % with $N$ as the third largest $|n_j|$.
	\end{theorem}
	
\begin{proof}
The proof goes exactly as the one in Lemma 2.2 of \cite{FGL} with $L^{p}$ playing the role of $\rho^2$
and $p$ the one of $\lambda$ in their notations.
\end{proof}

If the Nonresonance condition is fulfilled, then we can conclude that for appropriate $\lambda$
\begin{align*}
\sum_{q \geq 1} \lambda_q \omega_q= 0
\end{align*}
implies $\lambda_q=0$ for all $q$ and
\begin{align*}
\sum_{q \geq 1} \lambda_q \omega_q \neq 0
\end{align*}
implies 
\begin{align*}
\Big|\sum_{q \geq1} \lambda_q \omega_q \Big| \geq  \frac{\gamma}{N^{\tau}}.
\end{align*}

			\subsection{Normal Form Theorem}
Now we state the normal form theorem in \cite{BG} which we shall use in order
to prove our main result.

\begin{theorem} \label{normalform}
Consider the equation
\begin{align} \label{transe}
 i\dot{x}=  \omega  x+ \sum_{k\geq 2} f_k(x).
\end{align}
and assume the nonresonance condition \eqref{nrcond}.
For any $\ell \in \mathbb{N}$, there exists ${s}_0={s}_0(\ell,\tau)$ such that for any $s\geq{s}_0$
there exists $r_{s}>0$ such that for $r<r_s$, there exists an analytic canonical change of variables
\begin{align*} 
y=\Phi^{(\ell)}(x)\\
\Phi^{(\ell)}: B_s( r) \rightarrow B_s(3r)
\end{align*}
which puts \eqref{transe} into the normal form
\begin{align} \label{normal}
i\dot{y} = \omega y +\mathcal{R}^{(\ell)}(y) %+\mathcal{Y}^{(\ell)}(y)
+ \mathcal{X}^{(\ell)}(y).
\end{align}
Moreover %for $R < r_{\ell}/N^{\tau}$, with $\tau$ the one in \eqref{nrcond},
there exists a constant $C=C_s$ such that:
\begin{itemize}
\item \begin{align*} \sup_{x \in B_s(r)} \|x-\Phi^{(\ell)}(x)\|_s \leq C r^2 \end{align*}
\item $\mathcal{R}^{(\ell)}$ is at most of degree $\ell+2$, is resonant, and has tame modulus
\item the following bound holds
%$\mathcal{X}^{(\ell)}$ has tame modulus, has a zero of order $\ell+2$ at the origin, and
%\footnote{Here
%and henceforth we write $a\lesssim b$ whenever there exists a constant $c$ depending only
%on $d,\ell,\gamma$ such that $a\leq cb$.}
\begin{align*} \|\mathcal{X}^{(\ell)}\|_{s,r} \leq C r^{\ell+\frac{3}{2}} \end{align*}
%\item $\mathcal{Y}^{(\ell)}$ has tame modulus and \begin{align*} \|\mathcal{Y}^{(\ell)}\|_{s,r_{\ell}} \leq C
% \frac{R^2}{N^{s-1}} \end{align*}
\end{itemize}
\end{theorem}
			\subsection{Normal terms}
  Let's further characterize the resonant terms by starting with the case $d=1$.
\begin{lemma}[One-dimensional Case] \label{1dres}
Fix $K \in \mathbb{N}$.  Consider $\ad_{\omega}$ as a function on homogeneous vector-valued polynomials of degree $K$. The degree $K$ resonant terms of equation \eqref{normal} are of the form
\begin{align*}
P_m(\{|y_n|^2\},\{y_n\bar{y}_{-n}\}) y_m \partial_{y_m}+ Q_m(\{|y_n|^2\},\{y_n\bar{y}_{-n}\})y_{-m}\partial_{y_m}
\end{align*}
where $P_m \in \mathbb{R}$ and $\bar{Q}_m=Q_{-m}$.
\end{lemma}
\begin{proof}
A  degree $K$ resonant monomial is of the form $y^{\alpha}\bar{y}^{\beta}\partial_{y_m}$, where $\alpha, \beta$ and $m$ satisfy
\begin{align*}
(\alpha-\beta)\cdot \Omega -\Omega_m&=0
\end{align*}
which can be rewritten as
\begin{align*}
\sum_{i=1}^{M}\Omega_{n_i} - \sum_{i=1}^{M-1} \Omega_{k_i} -\Omega_m&=0
\end{align*}
for some $M>0$.
The nonresonance condition on the eigenvalues $\Omega_n$ implies that $(\alpha-\beta)\cdot \Omega -\Omega_m=0$ is (possibly up to reordering) equivalent to 
\begin{align*}
\Omega_{n_i}=\Omega_{k_i}, \hspace{.5cm} \Omega_{m}=\Omega_{n_M}\\
\Leftrightarrow |n_i|=|k_i|, \hspace{.5cm} |m|=|n_M|.
\end{align*}
On the other hand, the conservation of momentum provides the following relation:
\begin{align*}
\sum_{i=1}^M n_i- \sum_{i=1}^{M-1} k_i -m=0
\end{align*}
In other words
the system of equations,
\begin{align*}
|n_i|=|k_i|, \hspace{.5cm} |m|=|n_M|\\
\sum^M_{i=1} n_i- \sum^{M-1}_{i=1} k_i -m=0
\end{align*}
characterizes the resonant terms.  We will break up the characterization into cases.  The first case is when $m=n_M$ and we have:
\begin{equation}\label{case1}
\begin{aligned}
|n_i|=|k_i|, \hspace{.5cm} m=n_M \\
\sum^{M-1}_{i=1} n_i- \sum^{M-1}_{i=1} k_i =0.
\end{aligned}
\end{equation}
For $n_{i},k_i$ satisfying \eqref{case1} above, there exists
$S=S(k_1,\ldots,k_{M-1})\geq0$ such that we can write \eqref{case1} as
\begin{equation}\label{case1a}
\begin{aligned}
n_{i_1}=-k_{i_1},\ldots,n_{i_S}=-k_{i_S}, \\
n_{i_{S+1}}=k_{i_{S+1}},\ldots,n_{i_{M-1}}=k_{i_{M-1}} \\ m=n_M \\
\sum^{M-1}_{i=1} n_i- \sum^{M-1}_{i=1} k_i =0.
\end{aligned}
\end{equation}
%Now let $S:=\{i \in \{1,\ldots, M-1\}\,\left|\, n_i=-k_i \right.\}$. 
From the equation $n_{i_1}=-k_{i_1},\ldots,n_{i_S}=-k_{i_S}$ the resonant term contains a factor of the form
\begin{align*}
\prod_{1 \leq j \leq S } y_{n_{i_j}}\bar{y}_{-n_{i_j}}
\end{align*}
where we note that $\sum_{j=1}^S n_{i_j}=0$. From $n_{i_{S+1}}=k_{i_{S+1}},\ldots,n_{i_{M-1}}=k_{i_{M-1}}$ we obtain
\begin{align*}
\prod_{S< j \leq M-1} |y_{n_{i_j}}|^2.
\end{align*}
The full resonant term corresponding to \eqref{case1a} will be
\begin{align*}
y_m \prod_{S< j \leq M-1} |y_{n_{i_j}}|^2\prod_{1 \leq j \leq S } y_{n_{i_j}}\bar{y}_{-n_{i_j}} \partial_{y_m}.
\end{align*}
Therefore the resonant terms for the case $m=n_M$ will be the sum over all $\{n_i,k_i\}$ that satisfy \eqref{case1a}, namely:
\begin{align}\label{sum}
\left(\sum_{0\leq S\leq M-1} \sum_{\substack{n_{i_1},\ldots,n_{i_{M-1}}\in\mathbb{Z}\\ \sum_{j=1}^S n_{i_j}=0}} \prod_{S< j \leq M-1} |y_{n_{i_j}}|^2\prod_{\substack{1 \leq j \leq S  \\ }} y_{n_{i_j}}\bar{y}_{-n_{i_j}}\right) y_m \partial_{y_m}
\end{align}
For each $n_{i_1},...,n_{i_{M-1}} \in \mathbb{Z}$ and $S\in \{0,...,M-1\}$, we observe that the condition $\sum_{1\leq j \leq S} n_{i_j}=0$ implies that the terms 
\begin{align*}
\prod_{S< j \leq M-1} |y_{n_{i_j}}|^2\prod_{1 \leq j \leq S } y_{n_{i_j}}\bar{y}_{-n_{i_j}} \mbox{ and }\prod_{S< j \leq M-1} |y_{n_{i_j}}|^2\prod_{1 \leq j \leq S }\bar{y}_{n_{i_j}}y_{-n_{i_j}}
\end{align*}
both appear in \eqref{sum} and thus
\begin{align*}
\left(\sum_{0\leq S\leq M-1} \sum_{\substack{n_{i_1},\ldots,n_{i_{M-1}}\in\mathbb{Z}\\ \sum_{j=1}^S n_{i_j}=0}} \prod_{S< j \leq M-1} |y_{n_{i_j}}|^2\prod_{\substack{1 \leq j \leq S  \\ }} y_{n_{i_j}}\bar{y}_{-n_{i_j}}\right) \in \mathbb{R}
\end{align*}
and we define
\begin{align*}
P_m(\{|y_n|^2\},\{y_n\bar{y}_{-n}\}) := \sum_{0\leq S\leq M-1} \sum_{\substack{n_{i_1},\ldots,n_{i_{M-1}}\in\mathbb{Z}\\ \sum_{j=1}^S n_{i_j}=0}} \prod_{S< j \leq M-1} |y_{n_{i_j}}|^2\prod_{\substack{1 \leq j \leq S  \\ }} y_{n_{i_j}}\bar{y}_{-n_{i_j}}
\end{align*}

We now consider the case $-m=n_M$.  With this assumption we have
\begin{align*}
|n_i|=|k_i|, \hspace{.5cm} -m=n_M\\
\sum^{M-1}_{i=1} n_i- \sum^{M-1}_{i=1} k_i =2m
\end{align*}
and by the same argument the resonant terms from this case will be
\begin{align*}
\left(\sum_{0\leq S\leq M-1} \sum_{\substack{n_{i_1},\ldots,n_{i_{M-1}}\in\mathbb{Z}\\ \sum_{j=1}^S n_{i_j}=m}} \prod_{S< j \leq M-1} |y_{n_{i_j}}|^2\prod_{\substack{1 \leq j \leq S  \\ }} y_{n_{i_j}}\bar{y}_{-n_{i_j}}\right) y_{-m} \partial_{y_m}
\end{align*}
We now define 
\begin{align*}
Q_m(\{|y_n|^2\},\{y_n\bar{y}_{-n}\}):= \sum_{0\leq S\leq M-1} \sum_{\substack{n_{i_1},\ldots,n_{i_{M-1}}\in\mathbb{Z}\\ \sum_{j=1}^S n_{i_j}=m}} \prod_{S< j \leq M-1} |y_{n_{i_j}}|^2\prod_{\substack{1 \leq j \leq S  \\ }} y_{n_{i_j}}\bar{y}_{-n_{i_j}}
\end{align*}
and note that
\begin{align*}
\bar{Q}_m&= \sum_{0\leq S\leq M-1} \sum_{\substack{n_{i_1},\ldots,n_{i_{M-1}}\in\mathbb{Z}\\ \sum_{j=1}^S n_{i_j}=m}} \prod_{S< j \leq M-1} |y_{n_{i_j}}|^2\prod_{\substack{1 \leq j \leq S  \\ }} \bar{y}_{n_{i_j}}y_{-n_{i_j}}\\
&= \sum_{0\leq S\leq M-1} \sum_{\substack{n_{i_1},\ldots,n_{i_{M-1}}\in\mathbb{Z}\\ \sum_{j=1}^S n_{i_j}=-m}} \prod_{S< j \leq M-1} |y_{n_{i_j}}|^2\prod_{\substack{1 \leq j \leq S  \\ }} y_{n_{i_j}}\bar{y}_{-n_{i_j}}= Q_{-m}
\end{align*}
\end{proof}
The first step of analyzing the dynamical characteristics of the resonant terms is observing that the linear and resonant parts of the normalized Hamiltonian can be decoupled as a family of self-adjoint matrices.
\begin{corollary} \label{selfadjoint}
The truncation of \eqref{normal},
\begin{align*}
\dot{y}= \omega y + \mathcal{R}^{(\ell)}(y) 
\end{align*}
can be decoupled in the following way:
\begin{align} \label{action}
i\partial_t \left( \begin{array}{c} y_n \\ y_{-n} \end{array} \right) = \mathcal{M}_n  \left( \begin{array}{c} y_n \\ y_{-n} \end{array} \right) 
\end{align}
where $\mathcal{M}_n=\mathcal{M}_n\left(\omega, \{|y_m|^2\}, \{y_m\bar{y}_{-m}\}\right)$ is a self-adjoint matrix for all $t$.
\end{corollary}
\begin{proof}
We Taylor expand $\mathcal{R}^{(\ell)}$:
\begin{align*}
\mathcal{R}^{(\ell)}=\sum_{i=2}^{\ell+2}c_i\mathcal{R}_i
\end{align*}
where each $c_i$ is a multiplicity constant and $\mathcal{R}_i$ is homogeneous and degree $i$.  In fact, from Lemma \ref{1dres}
\begin{align*}
\mathcal{R}_i = \sum_{m\in \mathbb{Z}} \big(P^{(i)}_m(\{|y_n|^2\},\{y_n\bar{y}_{-n}\}) y_m + Q^{(i)}_m(\{|y_n|^2\},\{y_n\bar{y}_{-n}\})y_{-m}\big)\partial_{y_m}.
\end{align*}
Now we define $\mathcal{M}_n$ with the following components
\begin{align*}
(\mathcal{M}_n)_{11}&:= \omega_{|n|} + \sum_{i=2}^{\ell+2} c_i P^{(i)}_m  &(\mathcal{M}_n)_{12}:= \sum_{i=2}^{\ell+2} c_i Q^{(i)}_m\\
(\mathcal{M}_n)_{21}&:=  \sum_{i=2}^{\ell+2} c_i Q^{(i)}_{-m} &(\mathcal{M}_n)_{22}:= \omega_{|n|} + \sum_{i=2}^{\ell+2} c_i P^{(i)}_{-m}
\end{align*}
It follows immediately that $\mathcal{M}_n$ is self-adjoint.
\end{proof}
The one-dimensional case is included for instructive purposes. We can extend that arguments directly to the case $d>1$.  It will be even more evident that the form of the resonant terms depends entirely on two properties of the Hamiltonian that have been mentioned previously:
\begin{itemize}
\item  The Hamiltonian obeys the Conservation of Momentum law and 
\item  $\{ \omega_q\}_{ q<N}$ is a linearly independent set
\end{itemize}
\begin{lemma}[Higher-dimensional Case]
The resonant terms are of the form
\begin{align*}
\ker \ad_{\omega }=\mbox{span}\left(
\big\{Q_{m,i}y_{i}\partial_{y_m} \,\big|\, i , m \in \mathbb{Z}^d \big\} \right)
\end{align*}
where $\bar{Q}_{m,i}=Q_{i,m}$ and $Q_{i,m}=0$ when $|i|\neq |m|$.
In particular, $Q_{m,i}$ depends on $y_m$ only through terms of the form $y_{n}\bar{y}_k$ with $|n|=|k|$.
\end{lemma}
\begin{proof}
Just as in the one-dimensional case, the general resonant monomial of degree $k$ is of the form $y^{\alpha}\bar{y}^{\beta}\partial_{y_m}$, where $\alpha, \beta$ and $m$ satisfy
\begin{align*}
(\alpha-\beta)\cdot \Omega -\Omega_m&=0\\
\sum_{i=1}^{M}\Omega_{n_i} - \sum_{i=1}^{M-1} \Omega_{k_i} -\Omega_m&=0
\end{align*}
for some $M$ such that $2M-1=k$.
The nonresonance condition on the eigenvalues $\omega_q$ implies that, possibly up to reordering,
$(\alpha-\beta)\cdot \Omega -\Omega_m=0$ is equivalent to 
\begin{align*}
\Omega_{n_i}=\Omega_{k_i}, \hspace{.5cm} \Omega_{m}=\Omega_{n_M} \\
\Leftrightarrow |n_i|=|k_i|, \hspace{.5cm} |m|=|n_M|
\end{align*}
Conservation of momentum provides the following relation:
\begin{align*}
\sum_{i=1}^{M} n_i- \sum_{i=1}^{M-1} k_i -m=0
\end{align*}
The system of equations,
\begin{align*} 
|n_i|=|k_i|, \hspace{.5cm} |m|=|n_M|\\
\sum^M_{i=1} n_i- \sum^{M-1}_{i=1} k_i -m=0
\end{align*}
characterizes the resonant terms.  We will again break up the characterization into cases.  The first case is when $m=n_M$ and we have:
\begin{align}\label{highdsys}
|n_i|=|k_i|, \hspace{.5cm} m=n_M\\
\sum^{M-1}_{i=1} n_i- \sum^{M-1}_{i=1} k_i =0 \nonumber
\end{align}
and the resonant term corresponding to these equations will be of the form
\begin{align*}
y_m \prod_{i \in \{1,...,M-1\}}
y_{n_i}\bar{y}_{k_i} \partial_{y_m}
\end{align*}
At degree $2M-1$, the resonant terms for the case $m=n_M$ will be the sum over all $\{n_i,k_i\}$ that satisfy \eqref{highdsys}:
\begin{align}\label{hsum}
\left( \sum_{\sum n_i-\sum k_i=0} \prod_{i \in \{1,...M-1\}} y_{n_i}\bar{y}_{k_i}\right) y_m \partial_{y_m}
\end{align}
We observe that the condition $\sum_{i \in \{1,...M-1\}} n_i-k_i=0$ implies that for each $\{n_i,k_i\}$ the terms 
\begin{align*}
 \prod_{i \in \{1,...M-1\}} y_{n_i}\bar{y}_{k_i} \mbox{ and }  \prod_{i \in \{1,...M-1\}} y_{k_i}\bar{y}_{n_i}
\end{align*}
both appear in 	\eqref{hsum} and thus
\begin{align*}
\left( \sum_{\sum n_i-\sum k_i=0} \prod_{i \in \{1,...M-1\}} y_{n_i}\bar{y}_{k_i}\right)  \in \mathbb{R}
\end{align*}
and we define
\begin{align*}
Q_{m,m} := \left( \sum_{\sum n_i-\sum k_i=0} \prod_{i \in \{1,...M-1\}} y_{n_i}\bar{y}_{k_i}\right)
\end{align*}

For the case $n_M\neq m$.  With this assumption we have
\begin{align*}
|n_i|=|k_i|, \hspace{.5cm} m^*:=n_M\\
\sum^{M-1}_{i=1} n_i- \sum^{M-1}_{i=1} k_i =m-m^*
\end{align*}
and by the same argument the resonant terms from this case will be
\begin{align*}
\left( \sum_{\sum n_i-\sum k_i=m-m^*} \prod_{i \in \{1,...M-1\}} y_{n_i}\bar{y}_{k_i}\right) y_{m^*} \partial_{y_m}
\end{align*}
We now let 
\begin{align*}
Q_{m,m*}:= \sum_{\sum n_i-\sum k_i=m-m^*} \prod_{i \in \{1,...M-1\}} y_{n_i}\bar{y}_{k_i}
\end{align*}
and note that
\begin{align*}
\bar{Q}_{m,m^*}=\sum_{\sum n_i-\sum k_i=m-m^*} \prod_{i \in \{1,...M-1\}} \bar{y}_{n_i}y_{k_i}=\sum_{\sum n_i-\sum k_i=m^*-m} \prod_{i \in \{1,...M-1\}} y_{n_i}\bar{y}_{k_i}= Q_{m^*,m}
\end{align*}
\end{proof}
The analogue to Corollary \ref{selfadjoint} is as follows:
\begin{corollary} \label{highdselfadjoint}
The truncation of \eqref{normal},
\begin{align*}
i\dot{y}= \omega y + \mathcal{R}^{(\ell)}(y) 
\end{align*}
can be decoupled in the following way:
\begin{align} \label{action}
i\partial_t \left( \begin{array}{c} y_{n_1} \\ \cdots \\ y_{n_k} \end{array} \right) = \mathcal{M}_q  \left( \begin{array}{c} y_{n_1}\\ \cdots \\ y_{n_k} \end{array} \right) 
\end{align}
where $q>0$, $\{n_1,\ldots,n_k\}:=\{n\in\mathbb{Z}^d\,:\,|n|=q\}$, $\mathcal{M}_q=\mathcal{M}_q\left(\omega, \{y_j\}\right)$ is a self-adjoint matrix for all $t$.
\end{corollary}
\begin{proof}
As in \ref{selfadjoint}, we expand $\mathcal{R}^{(\ell)}$,
\begin{align*}
\mathcal{R}^{(\ell)}= \sum_{i=2}^{\ell+2} c_i \mathcal{R}_i
\end{align*}
with 
\begin{align*}
\mathcal{R}_i = \sum_{m\in \mathbb{Z}} \big(\sum_{j \in \mathbb{Z}^d\atop |j|=|m|}Q^{(i)}_{m,j}(\{y_n\})y_{j}\big)\partial_{y_m}.
\end{align*}
In conclusion, we define the components of $\mathcal{M}_q$:
\begin{align*}
\big(\mathcal{M}_q \big)_{mj}:= \delta_{mj}\omega_q +\sum_{i=2}^{\ell+2}c_i Q^{(i)}_{m,j}.
\end{align*}
\end{proof}

			\subsection{Iterative Lemma}
We now present the inductive lemma that is used to produce Theorem \ref{normalform}.  First consider a general Hamiltonian $H=H_0+P$.  Expand $P$ in Taylor series up to order $\ell_0+2$:
\begin{align*}
P=P^{(1)}+ \mathscr{R}_*\\
P^{(1)}:= \sum_{i=1}^{\ell_0} P_i
\end{align*}
where $P_i$ is homogeneous of degree $i+2$ and $\mathscr{R}_*$ is the remainder of the Taylor expansion.
\begin{lemma}[Iterative Lemma, \cite{BG} Proposition 4.20, Corollary 4.21] \label{iterlemma}
Consider the Hamiltonian $H=H_0 + P^{(1)}+\mathscr{R}_*$, and fix $s\geq \frac{d+1}{2}$.  For any $\ell\leq \ell_0$, there exists a positive $R_0\ll1$, and for any $N>1$, there exists an analytic canonical transformation
\begin{align*}
\mathscr{T}^{(\ell)}: B_s\Big( \frac{R_0(2\ell_0-\ell)}{2N^{\tau}\ell_0}\Big) \rightarrow H^s
\end{align*}
which transforms $H$ into
\begin{align*}
H^{(\ell)}:= H \circ \mathscr{T}^{(\ell)}= H_0 + \mathscr{L}^{(\ell)}+ f^{(\ell)} + \mathscr{R}^{(\ell)}_N + \mathscr{R}^{(\ell)}_T+ \mathscr{R}_* \circ \mathscr{T}^{(\ell)}.
\end{align*}
Let $L=\frac{2\ell_0-\ell}{2\ell_0}$.  For any $R < R_0N^{-\tau}$, there exists a constant $C$ such that the following properties are fulfilled:
\begin{enumerate}
\item the transformation $\mathscr{T}^{(\ell)}$ satisfies
	\begin{align*}
	\sup_{z \in B_s(RL)} \|z-\mathscr{T}^{(\ell)}(z)\|_s \leq CN^{\tau}R^2;
	\end{align*}
\item $\mathscr{L}^{(\ell)}$ is a polynomial of degree at most $\ell+2$ and has tame modulus; it is resonant and has a zero of order three at the origin; $f^{(r)}$ is a polynomial of degree at most $\ell_0+2$ and has zero of order $\ell+3$ at the origin; moreover, the following estimates hold:
\begin{align*}
\langle X_{\mathscr{L}^{(\ell)}} \rangle_{s, RL}  &\leq C R^2 \\
\langle X_{f^{(\ell)}} \rangle_{s,RL} &\leq C R^2(RN^{\tau})^{\ell}; 
\end{align*}
\item the remainder terms, $\mathscr{R}^{(\ell)}_N$ and $\mathscr{R}^{(\ell)}_T$, have tame modulus and satisfy
\begin{align*}
\|X_{\mathscr{R}^{(\ell)}_T}\|_{s,RL} &\leq C (RN^{\tau})^{\ell_0+2},\\
\|X_{\mathscr{R}^{(\ell)}_N}\|_{s,RL} &\leq C R^2 N^{\frac{d+1}{2}-s},\\
\|X_{ \mathscr{R}_* \circ \mathscr{T}^{(\ell)}}\|_{s,RL} &\leq C (RN^{\tau})^{\ell_0+2}.
\end{align*}
\end{enumerate}
\end{lemma}
		\section{Dynamics}
Finally, we can state the dynamical consequences of the normal form transformation given by Theorem \ref{normalform} by characterizing the normal terms.  The characterization allows us to show, as in \cite{FGL}, that these terms preserve the super actions:
\begin{proposition} \label{cutoff}
Suppose $y \in H^s$ satisfies  \eqref{action}, then
\begin{align*}
\partial_t \|y\|^2_s= \partial_t \sum_{q \geq 1} \left(\sum_{|n_i|=q}|y_{n_i}|^2 \right)\langle q \rangle^{2s}=0
\end{align*}
\end{proposition}
\begin{proof}
Fix $q$, let $\{n_1,\ldots,n_k\}:=\{n\in\mathbb{Z}^d\,:\,|n|=q\}$ and define
\begin{align*}
v:= \left( \begin{array}{c} y_{n_1} \\ \cdots \\ y_{n_k} \end{array} \right)
\end{align*}
Then by Corollary \ref{selfadjoint}
\begin{align*}
\partial_t \sum_{|n_i|=q}|y_{n_i}|^2 &= \partial_t v \cdot {\bar{v}}=\dot{v} \cdot {\bar{v}}+v \cdot \dot{{\bar{v}}}\\
&= \left(-i\mathcal{M}_q v\right) \cdot {\bar{v}}+v \cdot \left( i\overline{\mathcal{M}}_q {\bar{v}} \right)\\
&=\left(-i\mathcal{M}_q v\right) \cdot {\bar{v}}+\left(i\overline{\mathcal{M}}^T_q v \right) \cdot  {\bar{v}} \\
&=0
\end{align*}
\end{proof}
Proposition \ref{cutoff} shows that for any $y(t)$ satisfying  the truncated equation
\begin{align*}
i\partial_t y= \omega y + \mathcal{R}^{(\ell)}(y)
\end{align*}
there is no transference of mass between ``shells", 
\begin{align*}
v_q:= \left( \begin{array}{c} y_{n_1} \\ \cdots \\ y_{n_k} \end{array} \right)
\end{align*}
although there may be transference between $y_n$ and $y_m$ when $|n|=|m|$.  The following theorem completes our analysis on the dynamics of our system and allows us to prove the quantitative aspect of the stability in the main theorem, Theorem \ref{mainthm}.

\begin{theorem}\label{dynam}
Suppose $y \in B_s(r)$ satisfies \eqref{normal} with $r$ small enough. Then there exists a constant $C=C_s$
\begin{align*}
\partial_t \|y\|_s \leq C r^{\ell+\frac{5}{2}}
\end{align*}
\end{theorem}
\begin{proof}
\begin{align*}
\partial_t \|y\|_s&=\partial_t \langle y, y \rangle_s\leq |\langle\mathcal{X}^{(\ell)}(y) ,y \rangle_s|+ 
|\langle y, \mathcal{X}^{(\ell)}(y) \rangle_s | \\
&\leq C \|\mathcal{X}^{(\ell)} \|_{s,r}\|y\|_s
\end{align*}
The conclusion follows from Theorem \ref{normalform}.
\end{proof}

		\section{Conclusion}
We conclude by assembling the proof of Theorem \ref{mainthm}.  Theorem \ref{dynam}, shows that given a solution $y$ to equation \eqref{normal} with $\|y(0)\|_s <3r$, then for all $0<t <r^{-(\ell+\frac{3}{2})}$, $\|y\|_s< C_sr$.  Assuming $r$ is small enough, Theorem \ref{normalform} implies that we then have the same bound for any $x$ that solves \eqref{transe} with $\|x(0)\|_s<r$.  Finally, the transformation that takes a solution of equation \eqref{quada} to equation \eqref{quadb} is a change of coordinates on vectors of the form
\begin{align*}
v_q=
\left( \begin{array}{c} y_{n_1} \\ \cdots \\ y_{n_{k_q}} \end{array} \right)
\end{align*}
 where $k_q =\# \{n \in \mathbb{Z}^d \,|\, |n|=q\}$.  The transformation then preserves $\|v_q\|_{2}$ and therefore preserves the $H^s$ norm.  Thus, the bound on $\|x\|_s$, for $x$ satisfying \eqref{transe}, can be applied to $\|v\|_s$, for $v$ satisfying \eqref{mainham}.  Since we obtain $v$ by a gauge change to $(u_n)_{n \neq 0}$ fulfilling \eqref{ham}, we have
\begin{align*}
\|\psi(\cdot,t)-\psi_0(t)\|_{H^s(\mathbb{T}^d)}=\|u\|_s <C r
\end{align*}
for $t<r^{-(\ell+\frac{3}{2})}$.
The condition  $1-2p\lambda L_0^{p}>0$  which implies $1-2p\lambda L^{p}>0$ when $\lambda>0$, is necessary so that $\Omega_n \in \mathbb{R}$ for all $n \in \mathbb{Z}^d$.

\appendix
\section{} \label{appA}
The natural way in which to first study the behavior of $\|u-w_0\|_s$ would be to translate \eqref{p-nls} by $w_0(x,t)= \varrho e^{i\varrho^{2p}t}$ as follows
\begin{align}\label{translate}
i\partial_t (\psi-w_0)=\Delta (\psi-w_0) - |\psi-w_0|^{2p}(\psi-w_0)
\end{align}
With respect to Fourier coefficients $(\psi_n)_{n\in \mathbb{Z}^d}$, the linear part of \eqref{translate} can be decoupled as
\begin{align}\label{percoeff}
i\partial_t \left( \begin{array}{c} \psi_n \\ \bar{\psi}_{-n} \end{array} \right) = \left( \begin{array}{cc} -|n|^2-(p+1)\varrho^{2p} & -p\varrho^{2(p-1)}w_0^2\\ p\varrho^{2(p-1)}\bar{w}_0^2 & |n|^2+(p+1)\varrho^{2p} \end{array} \right)   \left( \begin{array}{c} \psi_n \\ \bar{\psi}_{-n} \end{array} \right) 
\end{align}
Using Floquet's Theorem, we deduce that, since \eqref{percoeff} has periodic coefficients, there exists a change of variables that transforms \eqref{percoeff} into a system with constant coefficients.  Indeed, if we let $v_n := \psi_n e^{-i\rho^{2p}t}$, then \eqref{percoeff} becomes
\begin{align}
i\partial_t \left( \begin{array}{c} v_n \\ \bar{v}_{-n} \end{array} \right) = \left( \begin{array}{cc} -|n|^2-p\varrho^{2p} & -p\varrho^{2p}\\ p\varrho^{2p} & |n|^2+p\varrho^{2p} \end{array} \right)   \left( \begin{array}{c} v_n \\ \bar{v}_{-n} \end{array} \right) 
\end{align}
Note that at $n=0$ the linear system produces a solution that grows linearly with time. However,
 one has
\begin{align*}
\inf_{\varphi\in\mathbb{R}} \|e^{-i\varphi}&w_0(\cdot,t)-\psi(\cdot,t)\|^2_{H^s(\mathbb{T}^d)}\\
&\leq  \left||w_0(0)|-|\psi_0(0)| \right|^2 + \left||\psi_0(0)|-|\psi_0(t)| \right|^2+ \|\psi(
\cdot,t)-\psi_0(t)\|^2_{H^s(\mathbb{T}^d)}.
\end{align*}
From here we see that 
\begin{align*}
\left||w_0(0)|-|\psi_0(0)| \right|^2 \leq \|w_0(\cdot,0)-\psi(\cdot,0)\|^2_{H^s(\mathbb{T}^d)}
\end{align*}
and by conservation of the $L^2$ norm
\begin{align*}
\left||\psi_0(0)|-|\psi_0(t)| \right|^2 \leq \left||\psi_0(0)|^2-|\psi_0(t)|^2 \right|=\left|\sum_{n\neq 0}|\psi_n(0)|^2-\sum_{n\neq 0}|\psi_n(t)|^2 \right|\\
\leq \|\psi(\cdot,t)-\psi_0(t)\|^2_{H^s(\mathbb{T}^d)}+\|\psi(\cdot,0)-\psi_0(0)\|^2_{H^s(\mathbb{T}^d)}.
\end{align*}
We have shown that it is sufficient to prove that  
\begin{align*}
\|\psi(\cdot,t)-\psi_0(t)\|^2_{H^s(\mathbb{T}^d)} <\varepsilon C(N,s_0,\varepsilon_0)
\end{align*}
 for $t< \varepsilon^{-N}$.

\noindent
{\bf Acknowledgements}. I am indebted to M. Procesi for her hospitality and time at the Universit\'{a} Degli Studi di Roma "La Sapienza", my Ph.D. advisor, Professor W. Schlag, for posing the problem and making my stay in Rome possible and L. Corsi for many useful discussions and comments.

\end{document}